\documentclass[a4paper,draft,reqno,12pt]{amsart}
\usepackage[utf8]{inputenc}
\usepackage[T1,T2A]{fontenc}
\usepackage[english]{babel}
\usepackage{amsmath}
\usepackage{amssymb}
\usepackage{amscd}
\usepackage{amsthm}
\usepackage{euscript}
\usepackage[all]{xy}
\newtheorem{proposition}{Proposition}

\newtheorem{lemma}{Lemma}
\newtheorem{theorem}{Theorem}

\theoremstyle{definition}
\newtheorem{definition}{Definition}
\newtheorem{example}{Example}
\theoremstyle{remark}
\newtheorem {remark}{Remark}

\DeclareMathOperator{\Spec}{Spec}

\DeclareMathOperator{\Conv}{Conv}

\def\KK{{\mathbb K}}

\def\ZZ{{\mathbb Z}}

\def\NNN{{\mathcal N}}
\def\QQ{{\mathbb Q}}

\def\NNN{\mathcal{N}}

\def\mal{\! \cdot \!}
\renewcommand{\phi}{\varphi}
\renewcommand{\ge}{\geqslant}
\renewcommand{\le}{\leqslant}

\begin{document}
\date{}
\title[Normally located polyhedra]{Normally located polyhedra}
\author{Ivan Arzhantsev}
\address{Faculty of Computer Science, HSE University, Pokrovsky Boulevard 11, Moscow, 109028 Russia}
\email{arjantsev@hse.ru}

\maketitle
\begin{abstract}
Lattice polyhedra $Q_1$ and $Q_2$ with the same tail cone are said to be normally located if every lattice point in the Minkowski sum $Q_1+Q_2$ is the sum of lattice points from $Q_1$ and $Q_2$, respectively. We prove that if the normal fan of $Q_1$ refines the normal fan of $Q_2$, then there is a positive integer $k$ such that for any positive integer $s$
the polyhedra $skQ_1$ and $skQ_2$ are normally located. This result is based on an interpretation of the problem in terms of graded algebras and earlier results on surjectivity of the multiplicaiton map on homogeneous components. Also we provide an example of two lattice triangles $P$ and $Q$ on the plane such that for any positive integer $k$ the triangles $kP$ and $kQ$ are not normally located.
\end{abstract} 


\section{Introduction}
\label{sec1}

Let us consider the lattice $\ZZ^d$ and the rational vector space $\QQ^d$ generated by $\ZZ^d$. By a \emph{lattice polytope} we mean a convex polytope $P$ in $\QQ^d$ with vertices in $\ZZ^d$. Let us assume additionally that the lattice points in $P$ generate the lattice $\ZZ^d$; this can be achieved by replacing $\ZZ^d$ with a proper sublattice. 

It is easy to construct examples of two lattice polytopes $P$ and $Q$ such that the Minkowski sum $P+Q$ contains a lattice point that is not a sum of lattice points from $P$ and $Q$, respective; see, e.g., \cite{Oda}. Moreover, starting from dimension $3$ it may happen even when $P=Q$. 

\begin{definition}
A lattice polytope $P$ is called \emph{normal} if for every positive integer $s$ and every lattice point $z\in sP$ there are lattice points $z_1,\ldots,z_s\in P$ such that $z=z_1+\ldots+z_s$. 
\end{definition} 

Normal polytopes play an important role in many areas of modern mathematics, see~\cite{Gub} for a recent survey on this subject. In particular, such polytopes define integrally closed graded monoid algebras and projectively normal embeddings of projective toric varieties. 

Let us recall that a lattice polytope $P$ is \emph{smooth} if the primitive edge vectors at every vertex of $P$ form a basis of $\ZZ^d$. Smooth polytopes correspond to projective embeddings of smooth toric varieties. Oda's question~\cite[Section~1]{Gub} asks whether every smooth polytope is normal. This question is open in all dimensions $\ge 3$.

The proof of the following theorem may be found in~\cite[Proposition~1.3.3]{BGT} or \cite{EW}. 

\begin{theorem} \label{scnor}
For every lattice polytope $P$ in $\QQ^d$ the polytope $(d-1)P$ is normal. 
\end{theorem} 

In particular, in dimension $2$ any lattice polygon is normal. 

\smallskip

The aim of this paper is to generalize the property of normality to a pair of polytopes. 

\begin{definition}
Lattice polytopes $P$ and $Q$ in $\QQ^d$ are said to be \emph{normally located} if for every lattice point $z\in P+Q$ there are lattice points $z'\in P$ and $z''\in Q$ such that $z=z'+z''$. 
\end{definition}

Oda's Conjecture (see \cite{Oda} or \cite[Section~1]{HNPS}) states the following. Let $P$ and $Q$ be lattice polytopes. Assume that $P$ is smooth and the normal fan of $Q$ refines that of $P$. Then the polytopes $P$ and $Q$ are normally located. This is a generalization of Oda's question on normality of a smooth polytope.  

In~\cite[Theorem~1.1]{HNPS}, it is shown that if $P$ and $Q$ are lattice polygons such that the normal fan of $Q$ refines that of $P$, then $P$ and $Q$ are normally located. In this paper we prove a version of Oda's Conjecture (Theorem~\ref{mcrit}) which generalizes~\cite[Theorem~1.1]{HNPS} to higher dimensions. The result is obtained not only for polytopes, but also for polyhedra with the same tail cone.

\smallskip

The paper is organized as follows. In Section~\ref{sec2} we give an explicit example of two lattice triangles $P$ and $Q$ on the plane such that for any positive integer $k$ the triangles $kP$ and $kQ$ are not normally located. This example is intended to demonstrate that, unlike the normality property, the property of normal location cannot always be achieved just by rescaling two given polytopes.

In Section~\ref{sec3} we recall basic definitions and facts on polyhedra, their Minkowski sums and normal fans. Also we consider polyhedra that appear as fibers of a projection of the positive octant in a bigger lattice to a smaller lattice. Section~\ref{sec3-1} is devoted to interpretations of the objects defined above in terms of graded algebras and the multiplication map on homogeneous components. We recall the results of~\cite{AH} that allow to relate the property of normal location (up to scalar) of a pair of polyhedra in fibers over two points with the location of these points with respect to the so-called GIT-fan. Finally, in Section~\ref{sec4} we explain how to realize a pair of polyhedra via the fiber construction and prove our main result (Theorem~\ref{mcrit}).

\section{An example in dimension $2$}
\label{sec2}

Let us show that, in contrast to Theorem~\ref{scnor}, for two lattice triangles absence of the property of normal location is not just a question of scale. 

\begin{proposition} \label{exm} 
There are two lattice triangles $P$ and $Q$ in $\QQ^2$ such that for every positive integer $k$ the triangles $kP$ and $kQ$ are not normally located.
\end{proposition} 

\begin{proof}
Take $P=\Conv(a_1,a_2,a_3)$ and $Q=\Conv(b_1,b_2,b_3)$, where
\smallskip
$$
a_1=(165,0), \ a_2=(175,0), \ a_3=(0,385) \quad \text{and} \quad b_1=(0,0), \ b_2=(35,0), \ b_3=(0,77). 
$$
\smallskip
Equivalently, the triangle $P$ is given by inequalities
$$
7x+3y\ge 1155, \quad 11x+5y\le 1925, \quad y\ge 0
$$
and $Q$ is given by 
$$
11x+5y\le 385, \quad x\ge 0, \quad y\ge 0.  
$$
The Minkowski sum $P+Q$ is $\Conv(c_1,c_2,c_3,c_4)$, where
$$
c_1=(165,0), \quad c_2=(210,0), \quad c_3=(0,385), \quad c_4=(0,462). 
$$
The inequalites that determine $P+Q$ are
$$
7x+3y\ge 1155, \quad 11x+5y\le 2310, \quad x\ge 0, \quad y\ge 0. 
$$
It is easy to check that for every positive integer $k$ the point $s=(1,385k-2)$ is contained in $kP+kQ$. 
Assume that $s=p+q$, where $p\in kP\cap\ZZ^2$ and $q\in kQ\cap\ZZ^2$. Then only two cases are possible. 

\medskip

{\it Case 1}. \ Let $p=(1,a)$ and $q=(0,385k-2-a)$ for some non-negative integer $a$. Since $p$ is contained in $P$, we have
$$
7+3a\ge 1155 k \quad \text{and} \quad 11+5a\le 1925k. 
$$
These inequalities may be rewritten as
$$
a\ge 385k -2 - \frac{1}{3} \quad \text{and} \quad a \le 385k -2 - \frac{1}{5}.
$$
Since $a$ is integer, we come to a contradiction. 

\medskip

{\it Case 2}. \ Let $p=(0,a)$ and $q=(1,385k-2-a)$ for some non-negative integer $a$. Since $p$ is contained in $P$, we have
$$
3a\ge 1155 k \quad \text{and} \quad 5a\le 1925k,
$$
so
$$
a\ge 385 k \quad \text{and} \quad a\le 385 k. 
$$
We conclude that $a=385k$. Since $q$ lies in $Q$, we have
$$
385k-2-a=-2>0.
$$
These two contradictions complete the proof. 
\end{proof} 


\section{Generalities on polyhedra} 
\label{sec3} 

By a \emph{polyhedron} in $\QQ^d$ we mean the intersection of finitely many closed affine half spaces.  A \emph{lattice polyhedron} is a polyhedron in $\QQ^d$ whose vertices are in $\ZZ^d$. Note that a polytope can be defined as a bounded polyhedron in $\QQ^d$. For a polyhedron $Q$ we define its \emph{relative interior} as the set obtained by removing all proper faces from $Q$. Let us denote the relative interior of $Q$ by $Q^\circ$. Further, a \emph{cone} in $\QQ^d$ is the intersection of finitely many closed linear half spaces. A cone is \emph{pointed} if it contains no line. If a cone has dimension at least $2$, it is pointed if and only if it is generated by its one-dimensional faces.

The set of all polyhedra in $\QQ^d$ comes with a natural structure of a commutative semigroup: one defines the Minkowski sum of two polyhedra $Q_1$ and $Q_2$ to be the polyhedron
$$
Q_1+Q_2 := \{w_1+w_2; \ w_1\in Q_1, w_2\in Q_2\}.
$$
In the same way one may define the Minkowski sum of two arbitrary subsets in $\QQ^d$. 

Any polyhedron allows a Minkowski sum decomposition $Q=P+\sigma$, where $P$ is a polytope and $\sigma$ is a cone in $\QQ^d$. In this decomposition, the \emph{tail cone} 
$\sigma$ is unique; it is given by
$$
\sigma=\{w\in\QQ^d; \ w'+tw\in Q \ \text{for all} \ w'\in Q, t\in\QQ_{\ge 0}\}.
$$

A polyhedron with the tail cone $\sigma$ is called a \emph{$\sigma$-polyhedron}. For example, polytopes are precisely $\sigma$-polyhedra with $\sigma=\{0\}$. 

It is easy to check that for a fixed cone $\sigma$ the set of all $\sigma$-polyhedra forms a commutative semigroup with respect to the Minkowski sum. 

\bigskip 

Let us recall that with any $\sigma$-polyhedron $Q$ in $\QQ^d$ one may associate the normal fan $\NNN(Q)$: any point $v\in Q$ defines the cone $\tau(v)$ consisting of all linear functions
on $\QQ^d$ which reach their maximal value on $Q$ at the point $v$. The collection of cones $\NNN(Q)=\{\tau(v)\, |\, v\in Q\}$ is finite and it is a fan in a sense that a face of any cone in 
$\NNN(Q)$ is contained in $\NNN(Q)$ and the intersection of any two cones in $\NNN(Q)$ is a face of each of them. Moreover, all cones in $\NNN(Q)$ are pointed if and only if $Q$ has full dimension in $\QQ^d$. 

The \emph{support} of a fan $\NNN$ is the union of all cones in $\NNN$. The support of the normal fan $\NNN(Q)$ equals the dual cone
$$
\sigma^{\lor}:=\{ l \in (\QQ^d)^* \ | \ l|_\sigma \ge 0\}. 
$$

We say that a fan $\NNN_1$ \emph{refines} a fan $\NNN_2$, if every cone in $\NNN_1$ is contained in some cone in~$\NNN_2$. Let $\NNN_1$ and $\NNN_2$ be two fans with the same support. The \emph{coarsest common refinement} of $\NNN_1$ and $\NNN_2$ is the fan $\NNN$ with the same support, whose cones are $\tau_1\cap \tau_2$, where $\tau_1\in\NNN_1$ and $\tau_2\in \NNN_2$. 

It is well-known that for normal fans $\NNN(Q_1)$ and $\NNN(Q_2)$ of two $\sigma$-polyhedra $Q_1$ and $Q_2$ the coarsest common refinement $\NNN$ coincides with the normal fan 
$\NNN(Q_1+Q_2)$. 

\bigskip

Now let us consider a surjective homomorphism of lattices $\pi\colon \ZZ^n\to \ZZ^m$ and the induced linear map of vector spaces $\pi\colon \QQ^n\to \QQ^m$. We denote by 
$\QQ^n_{\ge 0}$ the cone of vectors in $\QQ^n$ with non-negative coordinates and let $C:=\pi(\QQ^n_{\ge 0})\subseteq\QQ^m$. Consider the cone 
$\sigma:=\pi^{-1}(0)\cap\QQ^n_{\ge 0}$. With any point $u\in C$ one associates the polyhedron $P(u):=\pi^{-1}(u)\cap\QQ^n_{\ge 0}$. 

\begin{lemma}
For any $u\in C$ the polyhedron $P(u)$ is a $\sigma$-polyhedron.
\end{lemma}

\begin{proof}
For a vector $w\in\QQ^n$, the condition $w'+tw\in P(u)$ for all $w'\in P(u)$ and all $t\in\QQ_{\ge 0}$ means that $w\in\QQ^n_{\ge 0}$ and
$u=\pi(w'+tw)=\pi(w')+\pi(tw)=u+t\pi(w)$. It is equivalent to $w\in\QQ^n_{\ge 0}$ and $\pi(w)=0$, or $w\in\sigma$.  
\end{proof} 

Clearly, for every $u\in C$ there is a positive integer $r$ such that $P(ru)$ is a lattice polyhedron. Also it is easy to check that $P(u_1)+P(u_2)$ is contained in $P(u_1+u_2)$ for all
$u_1,u_2\in C$. 

\smallskip

We are interested in the following three properties of a pair $(u_1,u_2)$ with $u_1,u_2\in C$: 

\bigskip

(P1) \ \ $P(u_1)+P(u_2)=P(u_1+u_2)$; 

\bigskip

(P2) \ \ $(P(u_1) \cap \ZZ^n) + (P(u_2) \cap \ZZ^n)= P(u_1+u_2) \cap \ZZ^n$; 

\bigskip

(P3) \ \ there exists $k\in\ZZ_{>0}$ such that for any $s\in\ZZ_{>0}$ we have 
$$
(P(sku_1) \cap \ZZ^n) + (P(sku_2) \cap \ZZ^n)= P(sk(u_1+u_2)) \cap \ZZ^n.
$$ 

\section{Graded algebras} 
\label{sec3-1} 

In this section we introduce an algebraic interpretation of the objects discussed above. The projection $\pi\colon \ZZ^n\to \ZZ^m$ gives rise to an effective $\ZZ^m$-grading on the polynomial algebra $A:=\KK[x_1,\ldots,x_n]$. Namely, we put $\deg(x_i)=\pi(e_i)$, where $e_1,\ldots,e_n$ is the standard basis of the lattice $\ZZ^n$. For further purposes we assume the ground field $\KK$ to be an algebraically closed field of characteristic zero. 

\begin{remark}
The Linearization Problem~\cite{KR} claims that up to automorphism any effective $\ZZ^m$-grading on $\KK[x_1,\ldots,x_n]$ is obtained this way. 
\end{remark}

Below we follow the presentation given in~\cite{AH} in a somewhat more general situation. Let $A$ be an associative, commutative, integral, finitely generated algebra with unit over~$\KK$.
Suppose that $A$ is graded by the lattice $\ZZ^m$, i.e., we have
\begin{eqnarray*}
A
& = &
\bigoplus_{u \in \ZZ^m} A_u \ \ \ \  \text{with} \ \ \ \  A_{u_1}\cdot A_{u_2} \subseteq A_{u_1+u_2}. 
\end{eqnarray*}
By the \emph{weight cone} of $A$ we mean the cone $C(A) \subseteq \QQ^m$ generated by all $u \in \ZZ^m$ with $A_u \ne 0$. We investigate the following problem:
given $u_1,u_2 \in C(A) \cap \ZZ^m$, does there exist an integer $k > 0$ such that for any $s > 0$ the multiplication map 
$$ 
A_{sku_1} \otimes_\KK A_{sku_2}
\ \to \ 
A_{sk(u_1+u_2)},
\qquad
f \otimes g
\ \mapsto \
fg
$$
is surjective? We call a pair $u_1,u_2 \in C(A) \cap \ZZ^m$ \emph{generating} if it has this property. If $A$ is a polynomial algebra with $\ZZ^m$-grading given by the projection $\pi$, this is precisely property (P3) for a pair $(u_1,u_2)$ of lattice points in $C$. 

Let us  recall from~\cite{BH} the concept of the GIT-fan associated to a graded algebra. A~$\ZZ^m$-grading on $A$ defines an action of the torus $T := \Spec(\KK[\ZZ^m])$ on $X := \Spec(A)$ such that for any $u \in \ZZ^m$, the elements $f \in A_u$ are precisely the semiinvariants of the character $\chi^u \colon T \to \KK^*$, i.e., each $f \in A_u$ satisfies
\begin{eqnarray*}
f(t \mal x)
& := & 
\chi^u(t) f(x).
\end{eqnarray*}
The \emph{orbit cone} of a (closed) point $x \in X$ is the cone $\omega(x) \subseteq \QQ^m$ generated by all $u \in C(A)$ admitting an $f \in A_u$ with $f(x) \ne 0$.
The collection of orbit cones is finite, and thus one may associate to any element $u \in C(A)$ its \emph{GIT-cone}:
\begin{eqnarray*}
\lambda(u)
& := & 
\bigcap_{\substack{x \in X, \\ u \in \omega(x)}} \omega(x).
\end{eqnarray*}
These GIT-cones cover the weight cone $C(A)$ and, by~\cite[Theorem~3.11]{BH}, the collection $\Lambda(A)$ of all GIT-cones is a fan in the sense that if $\lambda \in \Lambda(A)$ then 
also every face of $\lambda$ belongs to $\Lambda(A)$, and for $\tau,\lambda \in \Lambda(A)$ the intersection $\tau \cap \lambda$ is a face of both $\lambda$ and $\tau$. 
Note that we allow here a fan to have cones containing lines.  

\begin{theorem} \cite[Theorem~1.1]{AH} \label{maina}
\begin{enumerate}
\item
If $u_1,u_2 \in C(A) \cap \ZZ^m$ is a generating pair, then the weights $u_1,u_2$ lie in a common  GIT-cone $\lambda \in \Lambda(A)$.
\item
If $u_1,u_2 \in C(A) \cap \ZZ^m$ lie in a common GIT-cone $\lambda \in \Lambda(A)$ and $u_1$ belongs  to the relative interior $\lambda^\circ$, then $(u_1,u_2)$ is a generating pair.
\end{enumerate}
\end{theorem}

If two weights $u_1,u_2 \in C(A) \cap \ZZ^m$ lie on the boundary of a common GIT-cone $\lambda \in \Lambda(A)$, then no general statement in terms of the GIT-fan is possible:
it may happen that $u_1,u_2$ is generating, and also it may happen that $u_1,u_2$ is not generating. For the first case there are obvious examples, and for the latter we present the following one.

\begin{example} \cite[Example~1.2]{AH} \label{oldex}
Consider the polynomial ring $A := \KK[x_1,x_2,x_3,x_4]$. Then one may define a $\ZZ^2$-grading on $A$ 
by setting
$$ 
\deg(x_1) \ := \ (4,1),
\quad
\deg(x_2) \ := \ (2,1),
\quad
\deg(x_3) \ := \ (1,2),
\quad
\deg(x_4) \ := \ (1,3).
$$
The pair $u_1 := (2,1)$ and $u_2 := (1,2)$ is contained in a common GIT-cone but it is not generating:
one checks directly that the monomials $x_1x_2^{s-2}x_3^{s-1}x_4 \in A_{s(u_1+u_2)}$ can never be obtained by multiplying elements from
$A_{su_1}$ and $A_{su_2}$. 
\end{example}

\begin{remark}
In~\cite[Theorem~1.5]{AH}, a criterion for a pair of weights $(u_1,u_2)$ in one GIT-cone to be generating is given in terms of normality of the image of a morphism between certain quotient spaces. 
\end{remark}

The proof of~\cite[Theorem~1.1]{AH} is based on several propositions. A modification of one of them will be used below. To formulate this modification, we need some more notions from the theory of graded algebras.

A subalgebra $B$ of a graded algebra $A = \oplus_{u \in \ZZ^m} A_u$ is \emph{homogeneous} if $B$ is the direct sum of intersections of $B$ with homogeneous components of $A$. Every homogeneous subalgebra in $A$ inherits a $\ZZ^m$-grading. 

With any weight $u\in\ZZ^m$ one associates a homogeneous subalgebra $A(u):=\oplus_{r\ge 0} A_{ru}$. Note that the subalgebra $A(u)$ is $\ZZ_{\ge 0}$-graded. 

Further, with any weights $u_1$ and $u_2$ one associates a homogeneous subalgebra $A(u_1,u_2)$ in $A(u_1+u_2)$ defined as
$$
A(u_1,u_2):=\bigoplus_{r\ge 0} A_{ru_1}\cdot A_{ru_2}. 
$$

We say that a homogeneous subalgebra $B$ of a $\ZZ_{\ge 0}$-graded algebra $A$ is \emph{big}, if the radical of the ideal $B_+:=\oplus_{r>0} B_r$ coincides with $A_+:=\oplus_{r>0} A_r$. 

\smallskip

Let $A$ be the polynomial algebra $\KK[x_1,\ldots,x_n]$ with $\ZZ^m$-grading given by a projection $\pi\colon\ZZ^n\to\ZZ^m$. For every $v\in\ZZ^n_{\ge 0}$ we denote by $x^v$ the monomial $x_1^{v_1}\ldots x_n^{v_n}$. 

\begin{proposition} \label{oneprop} 
Condition (P1) on a pair of weights $(u_1,u_2)$ is equivalent to each of the conditions:
\begin{enumerate}
\item[(1)]
the weights $u_1,u_2$ lie in a common  GIT-cone;
\item[(2)]
the subalgebra $A(u_1,u_2)$ is big in $A(u_1+u_2)$. 
\end{enumerate}
\end{proposition}

\begin{proof}
The equivalence of conditions (1) and (2) is proved in~\cite[Proposition~2.1]{AH}.

Let us prove that (P1) implies (2). Assume that $P(u_1+u_2)=P(u_1)+P(u_2)$ and take a monomial $x^v\in A(u_1+u_2)$. We have to prove that $x^v$ is contained in the radical of the ideal $A(u_1,u_2)_+$, i.e., there is a positive integer $q$ such that $x^{qv}\in A(u_1,u_2)_+$. By assumption, we have $v=v'+v''$ with $v'\in P(u_1)$ and $v''\in P(u_2)$. Then there is $q\in\ZZ_{>0}$ such that $qv=qv'+qv''$ with $qv'\in P(u_1)\cap\ZZ^n$ and $qv''\in P(u_2)\cap\ZZ^n$. This proves the assertion. 

Now we come to implication (2) $\Rightarrow$ (P1). Note that $v\in P(u_1+u_2)$ if and only if there is $s\in\ZZ_{> 0}$ such that $sv\in P(s(u_1+u_2))\cap\ZZ^n$ or, equivalently, the monomial $x^{sv}$ lies in $A(u_1+u_2)_+$. By assumption, this implies that $x^{sv}$ is contained in the radical of $A(u_1,u_2)_+$, or there exists $t\in\ZZ_{>0}$ such that
$x^{tsv}\in A(u_1,u_2)_+$. The later condition means that there are $v'\in P(tsu_1)\cap\ZZ^n$ and $v''\in P(tsu_2)\cap\ZZ^n$ such that $tsv=v'+v''$. This condition implies $v\in P(u_1)+P(u_2)$. 
\end{proof}


\section{Positive results on normal location} 
\label{sec4}

We keep the notation introduced in the previous section. In particular, every surjective homomorphism of lattices $\pi\colon \ZZ^n\to \ZZ^m$ gives rise to a $\ZZ^m$-grading on the algebra
$\KK[x_1,\ldots,x_n]$, and we speak about the weight cone, the orbit cones and the GIT-cones corresponding to this grading. In this situation, the weight cone coincides with 
$C=\pi(\QQ^n_{\ge 0})$ and the orbits cones are precisely the cones generated by all subsets of the set $\{w_1,\ldots,w_n\}$, where $w_i:=\pi(e_i)$. 

Let $\sigma$ be a pointed cone in $\QQ^d$. Taking an appropriate basis in $\QQ^d$ we may assume that $\sigma$ is contained in $\QQ^d_{\ge 0}$. Moreover, to any $\sigma$-polyhedra $Q_1$ and $Q_2$ we may apply a parallel translation and assume that $Q_1$ and $Q_2$ are contained in the open octant $\QQ^d_{>0}$. 

\begin{proposition} \label{ptwo}
Let $Q_1$ and $Q_2$ be two $\sigma$-polyhedra in $\QQ^d$ of full dimension. Then there are positive integers $n$ and $m$ with $d=n-m$ and a surjective homomorphism $\pi\colon \ZZ^n\to \ZZ^m$ such that $Q_1=P(u_1)$ and $Q_2=P(u_2)$ for some points $u_1,u_2\in C\cap\ZZ^m$ lying in a common GIT-cone. 
\end{proposition} 

\begin{proof}
Let $\NNN(Q_1)$ and $\NNN(Q_2)$ be the normal fans of the polyhedra $Q_1$ and $Q_2$, respectively. Denote by $\NNN$ the coarsest common refinement of the fans $\NNN(Q_1)$ and $\NNN(Q_2)$. 

Let $m$ be the number of rays of the fan $\NNN$ and $l_1,\ldots,l_m$ be linear functions on $\QQ^d$ generating the rays of $\NNN$. Denote by $a_i$ (resp. by $b_i$) the maximal value of $l_i$ on $Q_1$ (resp. on $Q_2)$. Then the polyhedron $Q_1$ (resp. $Q_2$) is given by inequalities $l_i\le a_i$ (resp. $l_i\le b_i$) with $i=1,\ldots,m$. The fan $\NNN$ is the normal fan of the polyhedron $Q_1+Q_2$, so the polyhedron $Q_1+Q_2$ is given by inequalities $l_i\le a_i+b_i$, $i=1,\ldots,m$. 

We may assume that the linear functions $l_1,\ldots,l_m$ have integer coefficients. Let $n=d+m$ and consider the projection $\pi\colon\ZZ^n\to\ZZ^m$ given by
$$
(x_1,\ldots,x_d,y_1,\ldots,y_m) \to (y_1 + l_1,\ldots,y_m + l_m). 
$$
Let $u_1=(a_1,\ldots,a_m)$. Then the polyhedron $P(u_1)$ is given by conditions
$$
x_1\ge 0,\ldots,x_d\ge 0,\ y_1\ge 0,\ldots, y_m\ge 0, \ y_1+ l_1=a_1,\ldots, y_m + l_m=a_m.
$$
This system is equivalent to 
$$
x_1\ge 0,\ldots,x_d\ge 0,\  l_1\le a_1,\ldots,l_m\le a_m. 
$$
It proves that $P(u_1)=Q_1$. The same arguments show that 
$$
P(u_2)=Q_2 \quad \text{and} \quad P(u_1+u_2)=Q_1+Q_2. 
$$
Finally, by Proposition~\ref{oneprop} the condition 
$P(u_1)+P(u_2)=P(u_1+u_2)$ implies that $u_1$ and $u_2$ lie in a common GIT-cone. 
\end{proof} 

Let us denote by $\NNN_1$ (resp. $\NNN_2$) the normal fan of the polyhedron $P(u_1)$ (resp. $P(u_2)$) living in the space $\QQ^d$. 

\begin{proposition} \label{int}
The fan $\NNN_1$ refines the fan $\NNN_2$ if and only if $u_1$ is an interior point of a GIT-cone containing $u_2$. 
\end{proposition}

\begin{proof}
For any point $c=(c_1,\ldots,c_n)\in\QQ^n$, we define a subset in $\{1,\ldots,n\}$ as
$$
Z(c)=\{ i \  | \ c_i=0 \}.
$$  
The condition that $u_1$ is an interior point of a GIT-cone containing $u_2$ means that every orbit cone, which contains $u_1$, contains $u_2$ as well. Note that the coordinates of a point
in $P(u_1)$ (resp. $P(u_2)$) may be considered as coefficients of a linear combination of the vectors $w_1,\ldots,w_n$ that is equal to $u_1$ (resp. $u_2$). Taking this into account, we may reformulate the above condition as: for every point $v_1\in P(u_1)$ there is a point $v_2\in P(u_2)$ such that $Z(v_1)\subseteq Z(v_2)$. 

Since we assume that the polyhedra are contained in the open octant $\QQ^d_{>0}$, the coordinates $x_1,\ldots,x_d$ are positive on $P(u_1)$ and  $P(u_2)$. At the same time, for $y_1,\ldots,y_m$ the coordinate $y_i$ is zero at a point $v_1$ (resp. $v_2$) if and only if the linear function $l_i$ reaches its maximum $a_i$ on $P(u_1)$ at $v_1$ (resp. $b_i$ on $P(u_2)$ at $v_2$). The condition $Z(v_1)\subseteq Z(v_2)$ means that every function $l_i$ which is contained in the cone $\tau(v_1)$ of the normal fan $\NNN_1$ (see Section~\ref{sec3}) is also contained in the cone $\tau(v_2)$ of the normal fan $\NNN_2$. Since the rays of the fans $\NNN_1$ and $\NNN_2$ are generated by some of the vectors $l_i$ and every cone is uniquely determined by the set of rays it contains, the last condition means that every cone in $\NNN_1$ is contained in a cone of $\NNN_2$ or, equivalently, the fan $\NNN_1$ refines the fan $\NNN_2$. 
\end{proof} 

The next theorem generalizes~\cite[Theorem~1.1]{HNPS} from dimension $2$ to higher dimensions.  

\begin{theorem} \label{mcrit}
Let $Q_1$ and $Q_2$ be two $\sigma$-polyhedra in $\QQ^d$ of full dimension.  Assume that the normal fan $\NNN(Q_1)$ refines the normal fan $\NNN(Q_2)$. Then there exists a positive integer $k$ such that for every positive integer $s$ the polyhedra $skQ_1$ and $skQ_2$ are normally located. 
\end{theorem}

\begin{proof}
Following Proposition~\ref{ptwo}, we realize $Q_1$ (resp. $Q_2$) as $P(u_1)$ (resp. $P(u_2)$). By Theorem~\ref{maina}, it suffices to show that the point $u_1$ lies in the relative interior
of a common GIT-cone of the points $u_1$ and $u_2$. It follows from Proposition~\ref{int}. 
\end{proof} 

\begin{remark}
The construction of two triangles presented in the proof of Proposition~\ref{exm} is extracted from Example~\ref{oldex}: in notation of this example, we take the triangles $P=P(u_1)$ and $Q=P(u_2)$ with respect to the projection $\pi\colon\ZZ^4\to\ZZ^2$.  
\end{remark}


\end{document}